\documentclass[11pt,twoside]{amsart}
\usepackage{amsmath,amssymb,amsthm}
\usepackage{xcolor}
\usepackage[makeroom]{cancel}

\newtheorem{thm}{Theorem}[section]

\newtheorem{lem}{Lemma}[section]
\newtheorem{prop}{Proposition}[section]

\def\Id{{\rm Id}\,}
\def\d{\partial}

\def\tilde{\widetilde}

\def\wt{\widetilde}

\newcommand{\bFormula}[1]{\begin{equation} \label{#1}}
	\newcommand{\eF}{\end{equation}}

\newcommand{\Div}{{\rm div}_x}
\newcommand{\Grad}{\nabla_x}
\newcommand{\Curl}{{\rm curl}_x}
\newcommand{\vr}{\varrho}

\newcommand{\bTheorem}[1]{\begin{Theorem} \label{T#1} }
	\newcommand{\eT}{\end{Theorem}}
\newcommand{\bProposition}[1]{\begin{Proposition} \label{P#1}}
	\newcommand{\eP}{\end{Proposition}}
\newcommand{\bLemma}[1]{\begin{Lemma} \label{L#1} }
	\newcommand{\eL}{\end{Lemma}}
\newcommand{\bCorollary}[1]{\begin{Corollary} \label{C#1} }
	\newcommand{\eC}{\end{Corollary}}

\newcommand\R{\mathbb{R}}

\newcommand{\N}{\mathbb{N}}
\newcommand{\ep}{\varepsilon}

\newcommand{\Tr}{\hbox{\rm{Tr}\,}}

\renewcommand{\div}{\mbox{\rm div}\;\!}

\def\dr{\delta\!\rho}
\def\dw{\delta\!w}

\def\cC{{\mathcal C}}

\def\cF{{\mathcal F}}

\newcommand{\with}{\quad\hbox{with}\quad}
\newcommand{\andf}{\quad\hbox{and}\quad}

\begin{document}
	
	\title
	[Global solutions for damped Euler-Maxwell equations]
	{Global solutions of Euler-Maxwell equations with dissipation}
	\author{ B. Ducomet \and \v S. Ne\v casov\' a \and J. S. H. Simon}
	
	\begin{abstract}
		We consider the Cauchy problem for a damped Euler-Maxwell system with  no ionic background.
		For smooth enough data  satisfying suitable so-called dispersive conditions, 
		we establish the global in time existence and uniqueness of a strong  solution that decays uniformly in time.
		Our method is inspired by the works  of D. Serre and M. Grassin %\cite{GS,S} 
		dedicated to the compressible Euler system.
	\end{abstract}
	\maketitle
	{\bf Keywords:} compressible Euler system, Maxwell, global solution, dissipation, decay.
	
	{\bf AMS subject classification:} 35Q30, 76N10.
	
	\section{Introduction}
	
	In recent years various proposal of including \textit{viscous} dissipative terms in Maxwell's equations have been made (see \cite{Ber,Mar,Ro,SL}) in order to accommodate propagation in lossy materials. 
	
	Compared to standard Maxwell's equations in lossless media, Maxwell's equations in a medium with conductor losses have numerous applications such as high-temperature plasmas, CPU electronic field or perfectly matched layers \cite{Ber}. In \cite{ME,SL,SQWS} this model leads to construction of efficient numerical methods for simulating electromagnetic dissipation.
	The possible derivation of such system from kinetic models, namely from the
	two species Vlasov-Boltzmann-Maxwell system, we refer to the work of Jang and Masmoudi, \cite{JM}. Also the hydrodynamics limits of the Boltzmann equation, see e.g. \cite{GS,LM}.
	
	In order to couple a dissipative Maxwell's system to a moving medium
	we consider the \textit{dissipative} 3D Euler-Maxwell system appearing naturally in plasma physics as a coupling between hydrodynamics and electromagnetism (see the derivation in e.g. \cite{CH}). The goal of the present paper is to study the associated Cauchy problem for this system.
	
	In fact, a number of works solve globally the Cauchy problem for the isentropic (or non isentropic) 3D Euler-Maxwell system, when a non-vanishing ionic background is present. Among them, one can quote \cite{D,GM,P,PWG,UK,UWK}. However, all of these works deal with a strictly positive ionic background, a crucial condition in order to get good estimates. 
	Hereafter, we are interested in the degenerate situation where one neglects ionic density and vacuum may appear. 
	Our aim is to establish the existence of a class of global solutions in that situation assuming some dissipation in the momentum and electro-magnetic equations (see \cite{PWG}).
	Recall that a similar issue has been investigated in the simpler situation 
	of the compressible Euler system in a series of papers  \cite{GS,S,G2} by D. Serre and M. Grassin. 
	There, it is  proved that for `well-prepared' data, the Cauchy problem for the compressible Euler system admits a unique global smooth solution.
	
	In our recent works  \cite{bddn,dd,ddbis}, we pointed out that Grassin-Serre strategy 
	was efficient to prove global existence results for the the Euler-Helmholtz, Euler-Poisson or Euler-Riesz systems.
	Our goal here is to adapt that strategy to the Euler-Maxwell system. 
	\medbreak
	After normalization, the system of equations to be studied for the fluid density $\vr = \vr (t,x)$, the charge density $\tilde\vr = \tilde\vr (t,x)$ (see \cite{I}), 
	the velocity field $u = u(t,x)$, 
	the electric field $E(t,x)$ and the magnetic field $B(t,x)$ as functions of the time $t$ and the Eulerian spatial coordinate $x \in \R^3$ reads:
	\bFormula{i1g}
	\partial_t{\vr} + \Div ({\vr} u) = 0,
	\eF
	\bFormula{i2g}
	\partial_t ({\vr} u) + \Div ({\vr} u \otimes u) + \Grad\Pi({\vr})={-\vr E+J\times B},
	%-\eta\vr  u,
	\eF
	\bFormula{i3g}
	\partial_t E-\Curl B=-J-\alpha_1 E,
	\eF
	\bFormula{i4g}
	\partial_t B+\Curl E=-\alpha_2 B,
	\eF
	\bFormula{i5g}
	\Div B=0,
	\eF
	\bFormula{i6g}
	\Div E=-\tilde\vr,
	\eF
	with initial data
	\bFormula{i0g}
	(\vr,\tilde\vr,u,E,B)(0,x)=(\vr_0,\tilde\vr_0,u_0,E_0,B_0)(x),
	\eF
	where $\alpha_1$ and $\alpha_2$ 
	%and $\eta$ 
	are given real positive parameters, $\Pi(\rho)=A\rho^\gamma$ is the barotropic pressure with $A>0$ and the adiabatic exponent $\gamma>1,$ and the electric current $J$ is given by Ohm's law:
	\bFormula{i7g}
	J=-\vr u.
	\eF
	It is known that one can discard equations \eqref{i5g} and \eqref{i6g} provided that they are satisfied by the data, namely
	\bFormula{i5gbis}
	\Div B_0=0,
	\eF
	\bFormula{i6gbis}
	\Div E_0=-\tilde\vr_0.
	\eF
	Finally we see that $\vr$ and $\tilde\vr$ are related by the compatibility relation
	\bFormula{i6gter}
	\partial_t(\vr-\tilde\vr) - \alpha_1\tilde\vr=0.
	\eF
	These conditions being assumed, 
	the reduced problem under study in the following is \eqref{i1g}--\eqref{i4g} supplemented with initial conditions  \eqref{i0g}.
	\vskip0.25cm
	Neglecting ionic background introduces the classical difficulty of vacuum as first pointed out  by Kato \cite{K} 
	when  symmetrizing the system for solving it. Despite this,  
	the corresponding Cauchy problem for Euler-Poisson with vacuum for strong solutions was solved locally in time in the eighties by various authors, among them:
	Makino  \cite{Ma}, Makino-Ukai \cite{MU}, Makino-Perthame \cite{MP},  Gamblin \cite{ga}, B\'ezard \cite{be}, Braun and Karp  \cite{BK}
	(see also \cite{Ma3} for a clear survey). 
	
	It is well known in this context that existence results are expected to be only local in time even for small data \cite{CW}
	(see blow-up results of Chemin \cite{C} (3D case) or Makino and Perthame \cite{MP} (1D spherically symmetric case)).
	However,   in  \cite{GS,G2,S}, D. Serre and M. Grassin pointed 
	out  that under a suitable   `dispersive'  spectral  condition  on the initial velocity that will be specified in the next section,
	and a smallness hypothesis on the initial density, the compressible Euler system 
	(that is  \eqref{i1g}--\eqref{i2g} with $E\equiv B\equiv0$) 
	admits a unique global smooth solution. 
	
	More recently we have shown \cite{bddn,dd,ddbis} that the Serre-Grassin global existence result extends to  the compressible Euler system coupled with the Poisson or Helmholtz equations. 
	Our goal here is to get a similar result for the whole (dissipative) Euler-Maxwell system  \eqref{i1g}--\eqref{i4g}.
	\bigbreak
	The rest of the paper is structured as follows. 
	In the next section, we state our main results and give some insights on our strategy. 
	In Section \ref{aux}, we establish decay estimates in Sobolev spaces first for 
	a  multi-dimensional Burgers-Maxwell system (which will provide us with  
	an approximate solution for our system) and, next, for  the compressible 
	Euler equation coupled with the Maxwell system.
	Section \ref{s:exist} is devoted to the proofs of  the main global existence 
	results, then we show the uniqueness of the solution.
	For the reader's convenience, some technical results like, in particular, first and second order commutator 
	estimates are recalled in the appendix. 
	\medbreak
	\noindent{\bf Notation:}
	Throughout the paper, $C$ denotes a harmless `constant' that may change from line
	to line, and we use sometimes  $A\lesssim B$ to mean that 
	$A\leq CB.$ The notation $A\approx B$  is used if  both $A\lesssim B$ and $B\lesssim A.$
	
	Finally, we shall denote by $\dot H^s$ and $H^s$ the homogeneous and nonhomogeneous
	Sobolev spaces of order $s$ on $\R^3,$ and by $W^{k,p}$ (with $k\in\N$ and $p\in[1,\infty]$)
	the set of $L^p$ functions on $\R^3,$  with derivatives up to order $k$ in $L^p.$

	\section{Main results}\label{main}
	
	Let us introduce the symmetrization introduced by T. Makino in \cite{Ma}, 
	setting
	\begin{equation}\label{eq:mak}
		\rho:=\frac{2\sqrt{A\gamma}}{\gamma-1}\,\varrho^{\frac{\gamma-1}2}.
	\end{equation}
	After that change of unknown, System \eqref{i1g}-\eqref{i4g} rewrites
	\begin{equation}\label{eq:MP}
		\left\{\begin{array}{l} 
			(\d_t+u\cdot\nabla)\rho+\frac{\gamma-1}2\rho\,\div u=0,\\[1.5ex]
			(\d_t+u\cdot\nabla)u+\frac{\gamma-1}2\rho\,\nabla\rho=-(E+u\times B)
			%-\eta u
			,\\[1.5ex]
			\partial_t E-\Curl B+\alpha_1 E=\vr u,\\[1.5ex]
			\partial_t B+\Curl E+\alpha_2 B=0.
		\end{array}
		\right.
	\end{equation}
	
	We consider the following generalized Burgers equation
	\bFormula{auxiso}
	\partial_t  v +  v \cdot \Grad  v = \Curl v\times v,
	%\Curl v\times v-\eta v,
	\eF
	complemented with the damping-free Maxwell System
	\begin{equation}
		\label{MS}
		\left\{\begin{array}{l} 
			\partial_t \overline E-\Curl \overline B+\alpha_1 \overline E=0,\quad \Div \overline{E} = 0,\\[1.5ex]
			\partial_t \overline B+\Curl \overline E+\alpha_2 \overline B=0,\quad \Div \overline{B} = 0,
		\end{array}
		\right.
	\end{equation}
	with initial conditions 
	\bFormula{auxiso0}
	v(0,x)=v_0(x),
	\eF
	and 
	\bFormula{auxiso00}
	(\overline E,\overline B)(0,x)=(\overline E_0,\overline B_0)(x).
	\eF
	%as a good approximation for Euler-Maxwell system (\ref{i1g})(\ref{i2g})(\ref{i3g})(\ref{i4g}).%\footnote{Should we also assume $\div_x \overline{E}= 0$ and $\div_x \overline{B}= 0$, so that \eqref{eq:compatibilty} holds true?}
	
	The impetus in considering such system is that, as we shall show later, \eqref{auxiso} is a good approximation of \eqref{i2g} provided that the initial density of the fluid and the initial electromagnetic field are small enough. 
	This can be done by extending the observation in \cite{G,GS} for the compressible Euler system and under suitable spectral conditions on the initial data. It is thus natural to expect that 
	system \eqref{auxiso}--\eqref{MS} is a good approximation of \eqref{eq:MP} provided that the initial density and the electromagnetic field are small.
	
	% Extending the observation in \cite{G,GS} for the compressible Euler system and under suitable spectral conditions on $Du_0,$ our goal is to show that
	%  equation \eqref{auxiso} is a good approximation  of \eqref{i1g} provided the density of the fluid and the electromagnetic field are small enough.

	The main strategy will be to justify this heuristics in order to construct a class of global solutions to our Euler-Maxwell system. 
	\medbreak
	%In order to precisely state our results, we need to introduce the following function space: 
	We consider the following function space for the analysis of the approximate equation \eqref{auxiso}:
	$$E^s:=\bigl\{z\in\cC(\R^3;\R^3),\; Dz\in L^\infty\text{ and }D^2z\in H^{s-2}\bigr\}\cdotp$$

	We now present the existence of a classical solution to equation \eqref{auxiso} given some spectral assumption on the data.  
	\begin{prop}
		\label{p:Burgers}
		Let $v_0$ be in $E^s$ with  $s>5/2$ and satisfy:
		$$
		\hbox{there exists }\ \varepsilon>0\ \hbox{ such that for any  }\  x\in\R^3,\ \hbox{\rm dist}(\hbox{\rm Sp}\,(Dv_0(x)),\R_-)\geq \varepsilon,
		\leqno\mathbf{(H0)}
		$$
		where  ${\rm Sp}\,A$ denotes the spectrum of the matrix $A$. 
		%$$
		%\hbox{the initial velocity gradient tensor is symmetric:}\ \ \left.\partial_iv_j%\right|_{t=0}=\left.\partial_jv_i\right|_{t=0}.
		%\leqno\mathbf{(H00)}
		% $$
		Then \eqref{auxiso} supplemented with \eqref{auxiso0} has a classical solution $v$ on $\R_+\times\R^3$ such that
		\[
		D^2v \in \cC^j\bigl(\R_+;H^{s-2-j}(\R^3)\bigr)\ \ \ \mbox{for}\ \ j=0,1.
		\]
		Moreover,  $D  v\in \cC_b(\R_+\times\R^3)$ and we have for any $t\geq 0$ and  $x\in \R^3,$
		\begin{equation}
			\label{Du}
			Dv(t,x)=\frac{1}{1+t}\ \Id +\frac{1}{(1+t)^2}\ K(t,x),
		\end{equation}
		for some function $K\in\cC_b(\R_+\times\R^3;\R^3\times\R^3)$ that also satisfies
		\begin{equation}
			\| K(t,\cdot)\|_{\dot H^\sigma}\leq K_\sigma (1+ t)^{\frac 12-\sigma},
			\label{estim2}
		\end{equation}
		for all $0<\sigma\leq s-1$. Moreover, there exists a constant $C>0$ such that $v$ satisfies the following estimates
		\begin{align}
			&\|Dv(t)\|_{L^\infty}\leq \frac{C}{1+t},\label{Duinfty}\\
			&\| v(t,\cdot)\|_{\dot H^{\sigma}}\leq C(1+ t)^{\frac 12-\sigma} , \label{Dlu}\\
			&\|D^2v(t)\|_{L^\infty}\leq \frac{C}{(1+ t)^3}, \label{estim3}
		\end{align}
		where $0<\sigma\leq s-1$.
		
	\end{prop}
	
	We mention that the proposition above has been established --- in the case where there right-hand side of \eqref{auxiso} is absent --- in \cite{G,GS} for integer regularity exponents, while X. Blanc, et al. \cite{bddn} provided the proof for \emph{real} exponents.
	
	Concerning the damping-free electromagnetic system we can establish an exponential decay of its solution as shown in the lemma below.
	\begin{lem}
		\label{EMfree}
		Assume that $\overline B_0,\overline E_0\in L^2$.
		Let $(\overline E,\overline B)$ be the unique solution of the Cauchy problem ~\eqref{MS}-\eqref{auxiso00}.
		Then $\overline E,\overline B\in L^2$ and $E$ satisfies the estimate
		\bFormula{Elec0}
		\|\overline E\|_{L^2}^2+\|\overline B\|_{L^2}^2\leq C_0^2e^{-\alpha t},
		\eF
		with $C_0^2=\frac{1}{2}(\|\overline E_0\|_{L^2}^2+\|\overline B_0\|_{L^2}^2)$ and $
		\alpha=\min\{\alpha_1,\alpha_2\}$.
	\end{lem}
	% {\bf Proof:} Multiplying the first equation (\ref{MS}) by $E$ and the second one by $B$, we get the energy equality
	% $$\frac{1}{2}\frac{d}{dt}\frac{1}{2} (B^2+ E^2)+\alpha_1 E^2+\alpha_2B^2
	% +\div (E\times B)=0.
	% $$
	% Integrating on $[0,t]\times {\mathbb R}^3$ we get
	% $$\int_{{\mathbb R}^3} \frac{1}{2} B^2 dx
	% +\frac{1}{2} \int_{{\mathbb R}^3} E^2\ dx+\int_0^t\int_{{\mathbb R}^3} \left(\alpha_1E^2(s)+\alpha_2B^2(s)\right)\ dx\ ds
	% =\int_{{\mathbb R}^3} \left(\frac{1}{2} B_0^2+\frac{1}{2} E_0^2\right) dx,$$
	% which gives the inequality
	% $$\frac{1}{2} \int_{{\mathbb R}^3} E^2\ dx+\alpha\int_0^t\int_{{\mathbb R}^3} (E^2(s)+B^2(s))\ dx\ ds\leq C_0^2.$$
	% Using Gronwall's inequality we get (\ref{Elec0}).
	% \hfill$\Box$
	\begin{proof}
		Multiplying the first equation (\ref{MS}) by $E$ and the second one by $B$, we get the energy equality
		$$\frac{1}{2}\frac{d}{dt}\frac{1}{2} (B^2+ E^2)+\alpha_1 E^2+\alpha_2B^2
		+\div (E\times B)=0.
		$$
		Integrating on $[0,t]\times {\mathbb R}^3$ we get
		$$\int_{{\mathbb R}^3} \frac{1}{2} B^2 dx
		+\frac{1}{2} \int_{{\mathbb R}^3} E^2\ dx+\int_0^t\int_{{\mathbb R}^3} \left(\alpha_1E^2(s)+\alpha_2B^2(s)\right)\ dx\ ds
		=\int_{{\mathbb R}^3} \left(\frac{1}{2} B_0^2+\frac{1}{2} E_0^2\right) dx,$$
		which gives the inequality
		$$\frac{1}{2} \int_{{\mathbb R}^3} E^2\ dx+\alpha\int_0^t\int_{{\mathbb R}^3} (E^2(s)+B^2(s))\ dx\ ds\leq C_0^2.$$
		Using Gronwall's inequality we get (\ref{Elec0}).
	\end{proof}
	\vskip0.25cm
	Going back to system ~\eqref{i1g}-\eqref{i7g}, we can derive an analogous $L^2$ bound for the actual electric field.
	
	\begin{lem}
		\label{Elecmag}
		Let $T>0$  be arbitrary. 
		Assume that $\sqrt{\vr_0} u_0,B_0,E_0\in L^2$ and $\vr\in L^\gamma$.
		Let $(\vr,u,E,B)$ be a solution of the Cauchy problem ~\eqref{i1g}-\eqref{i7g}.
		
		Then $E\in L^2$ and satisfies the estimate
		\bFormula{Elec}
		\|E\|_{L^2}^2+\|B\|_{L^2}^2\leq C_1^2e^{-\alpha t},
		\eF
		with $C_1^2=\|\sqrt{\vr_0} u_0\|_{L^2}^2+\frac{A\gamma}{\gamma-1}\|\vr_0\|_{L^\gamma}^\gamma+\frac{1}{2}(\|E_0\|_{L^2}^2+\|B_0\|_{L^2}^2)$.
	\end{lem}
	% {\bf Proof:} 
	% Multiplying (\ref{i2g}) by $u$, (\ref{i3g}) by $E$ and (\ref{i4g}) by $B$, we get the energy equality
	% $$\frac{d}{dt}\left(\frac{1}{2}\vr u^2+\frac{\Pi}{\gamma-1}\right)
	% +\frac{1}{2}\frac{d}{dt} (E^2+ B^2)+\alpha_1E^2+\alpha_2 B^2$$
	% $$+\div \left(\left(\vr u^2+\frac{\gamma\Pi}{\gamma-1}\right) u+E\times B\right)=0.
	% $$
	% Integrating on $[0,t]\times {\mathbb R}^3$ we get
	% $$\int_{{\mathbb R}^3} \left(\frac{1}{2}\vr u^2+\frac{\Pi}{\gamma-1}\right) dx
	% +\frac{1}{2} \int_{{\mathbb R}^3} (E^2+B^2)\ dx+\int_0^t\int_{{\mathbb R}^3} \left(\alpha_1E^2(s)+\alpha_2B^2(s)\right)\ dx\ ds$$
	% $$=\int_{{\mathbb R}^3} \left(\frac{1}{2}\vr_0 u_0^2+\frac{\Pi_0}{\gamma-1}+\frac{1}{2} B_0^2+\frac{1}{2} E_0^2\right) dx,$$
	% which gives the inequality
	% $$\frac{1}{2} \int_{{\mathbb R}^3} (E^2+B^2)\ dx+\alpha\int_0^t\int_{{\mathbb R}^3} (E^2(s)+B^2(s))\ dx\ ds\leq C_1^2.$$
	% Using Gronwall's inequality we obtain (\ref{Elec}).
	\begin{proof}
		Multiplying (\ref{i2g}) by $u$, (\ref{i3g}) by $E$ and (\ref{i4g}) by $B$, we get the energy equality
		$$\frac{d}{dt}\left(\frac{1}{2}\vr u^2+\frac{\Pi}{\gamma-1}\right)
		+\frac{1}{2}\frac{d}{dt} (E^2+ B^2)+\alpha_1E^2+\alpha_2 B^2$$
		$$+\,\div \left(\left(\vr u^2+\frac{\gamma\Pi}{\gamma-1}\right) u+E\times B\right)=0.
		$$
		Integrating on $[0,t]\times {\mathbb R}^3$ we get
		$$\int_{{\mathbb R}^3} \left(\frac{1}{2}\vr u^2+\frac{\Pi}{\gamma-1}\right) dx
		+\frac{1}{2} \int_{{\mathbb R}^3} (E^2+B^2)\ dx+\int_0^t\int_{{\mathbb R}^3} \left(\alpha_1E^2(s)+\alpha_2B^2(s)\right)\ dx\ ds$$
		$$=\int_{{\mathbb R}^3} \left(\frac{1}{2}\vr_0 u_0^2+\frac{\Pi_0}{\gamma-1}+\frac{1}{2} B_0^2+\frac{1}{2} E_0^2\right) dx,$$
		which gives the inequality
		$$\frac{1}{2} \int_{{\mathbb R}^3} (E^2+B^2)\ dx+\alpha\int_0^t\int_{{\mathbb R}^3} (E^2(s)+B^2(s))\ dx\ ds\leq C_1^2.$$
		Using Gronwall's inequality we obtain (\ref{Elec}).
	\end{proof}
	
	\vskip0.5cm
	
	Our main result is proving the following global existence and uniqueness of solution to System~\eqref{i1g}-\eqref{i4g}. 
	\begin{thm}
		\label{isentro}
		Suppose that either $1<\gamma<5/3$ and $5/2<s<3/2+2/(\gamma-1)$ or $5/2<s<+\infty$ if $\gamma=1+2/k$ for some integer $k$. 
		
		Assume that the initial data $(\rho_0,u_0,E_0,B_0)$ satisfy:
		\begin{itemize}
			\item $\mathbf{(H1)}$ there exists $v_0$ in $E^{s+1}$ satisfying  $\mathbf{(H0)}$
			and such that $u_0-v_0$ is small in $H^s$;
			\item $\mathbf{(H2)}$ $\vr_0^{\frac{\gamma-1}{2}}$ is small enough in $H^s.$
			\item $\mathbf{(H3)}$ $E_0$ and $B_0$ are small enough in $H^s.$
			\item $\mathbf{(H4)}$ Conditions \eqref{i5gbis}, \eqref{i6gbis} and $B_0=\Curl u_0$ are satisfied.
		\end{itemize}
		If $(v,\overline E,\overline B)$ is the global solution of \eqref{auxiso}--\eqref{MS} with initial data \eqref{auxiso0}--\eqref{auxiso00} from Proposition \ref{p:Burgers} and Lemma \ref{EMfree}, then there exists a unique global solution $\left(\vr,u,E,B\right)$ to
		\eqref{i1g}--\eqref{i0g}, such that 
		% Denoting by  $ v$ (resp. $(\overline E,\overline B)$)  the global solution of \eqref{auxiso}  (resp. \eqref{MS}) given by Proposition \ref{p:Burgers} (resp. Lemma \ref{EMfree}) with initial data \eqref{auxiso0}, there exists a unique global solution $\left(\vr,u,E,B\right)$ to
		% \eqref{i1g}-\eqref{i0g}, such that 
		\[
		\bigl(\vr^{\frac{\gamma-1}{2}},u- v,E-\overline E,B-\overline B\bigr) \in
		\cC\bigl(\R_+;H^{s}\bigr).
		\]
	\end{thm}
	As a consequence of the theorem above, we can show that the constructed global solution of \eqref{i1g}--\eqref{i0g} satisfies a decay estimate towards the solution of \eqref{auxiso}--\eqref{MS}.
	% The following decay estimates  are satisfied by  the global solutions constructed in the previous theorem (to be 
	% compared with those of Proposition \ref{p:Burgers}). 
	\begin{thm}\label{decay}
		Let all the assumptions of  Theorem \ref{isentro} be in force.  Then, for all $\sigma$ in $[0,s],$
		the solution $(\vr,u,E,B)$ constructed therein satisfies
		$$
		\left\|\vr^{\frac{\gamma-1}{2}},u- v,E-\overline E,B-\overline B\right\|_{\dot H^\sigma}\leq C_\sigma (1+t)^{3/2-\sigma-\min\{1,3/2(\gamma-1)\}},
		$$
		where $C_\sigma$ depends only on the initial data, on $\gamma$
		and on $\sigma.$ 
	\end{thm}

	%%%%%%%%%%%%%%%%%%%%%%%%%%%%%%%%%%%%%%%%%%%
	
	\section{Decay estimates in Sobolev spaces}\label{aux}
	
	In this section we prove \textit{a priori} decay estimates in Sobolev spaces which play a fundamental role in the proof of our global existence result. To be specific, we first establish a decay estimate for the generalized Burgers equation \eqref{auxiso}. Secondly, we compare the solutions to \eqref{eq:MP} and to \eqref{auxiso} giving us the estimate promised in Theorem \ref{decay}.
	% The goal of the present section is to prove  a priori decay estimates in Sobolev spaces first for the multi-dimensional generalized
	% Burgers equation \eqref{auxiso} and, next, for the discrepancy between the solution to  \eqref{eq:MP}  and to \eqref{auxiso}. Those estimates  will play a fundamental role 
	% in the proof of our global existence result.  
	
	\subsection{Decay estimates for the  generalized Burgers equation} 
	
	The purpose of this part is to prove Proposition \ref{p:Burgers} for  any  real regularity  exponent $s>5/2.$ 
	\smallbreak
	1. Using the identity $v\cdot\Grad v=\Grad \frac{v^2}{2}+\Curl v\times v$ in equation (\ref{auxiso}) we get
	$$\partial_t v+\Grad \frac{v^2}{2}
	%+\eta v
	=0.$$
	%Taking derivative $\partial_k$ we find
	%%$$\partial_t\partial_k v_i+\partial_k v_j\partial_i v_j+v_j%%\partial_k \partial_i v_j+\eta\partial_k v_i=0,$$
	%which rewrites
	%$$\partial_t\partial_i v_k+\partial_i v_j\partial_k v_j+v_j%\partial_i \partial_k v_j+\eta\partial_i v_k=0.$$
	%Subtracting we see that
	%$$\partial_t\left(\partial_k v_i-\partial_i v_k\right)+\eta%\left(\partial_k v_i-\partial_i v_k\right)=\partial_t\left(e^{\eta %t}\left(\partial_k v_i-\partial_i v_k\right)\right)=0.$$
	%Then the tensor ${\mathbb S}=Dv$ with ${\mathbb S}_{ij}=\partial_i %v_j$ is symmetric for any $t$ as soon as it is so at $t=0$.
	
	By letting $X$ be the flow of $v$, we see that the matrix valued function $A:(t,y)\mapsto Dv(t,X(t,y))$ satisfies the Ricatti equation
	$$A'+ A^2=0,\qquad A|_{t=0}=Dv_0.$$
	{}From Hypothesis $\mathbf{(H0)},$
	one can deduce that $v(t,y)$ is defined for all $t\geq0$ and $y\in\R^3,$ and that 
	$$Dv(t,X(t,y))= (\Id+t Dv_0(y))^{-1} Dv_0(y)\with  X(t,y)=y+tv_0(y).$$
	%where $\phi_\eta(t)=\frac{1-e^{-\eta t}}{\eta}$.
	
	Therefore, denoting $X_t:y\mapsto X(t,y),$ we get \eqref{Du}, that is
	\begin{align}
		\label{eq:Dv}
		Dv(t,x)=\frac{\Id}{1+t}+\frac{K(t,x)}{(1+t)^2}
	\end{align}
	where $K(t,x):=(1+t)(\Id+tDv_0(X_t^{-1}(x)))^{-1}(Dv_0(X_t^{-1}(x))-\Id)$. From this, we can also find the divergence of the velocity field $v$ by taking the trace of the tensor $Dv$
	\begin{equation}\label{eq:divv}
		\div v(t,y+tv_0(y))=\frac d{1+t}+\frac{{\rm Tr}\,K(t,y+tv_0(y))}{(1+t)^2}\cdotp
	\end{equation}
	Furthermore,   Hypothesis $\mathbf{(H0)}$ implies that
	\begin{equation}\label{eq:DXbounded}
		\|(\Id + tDv_0)^{-1}\|_{L^\infty}\lesssim (1+\ep t)^{-1},\end{equation}
	and  $K$ is thus bounded on $\R_+\times\R^3.$
	\medbreak
	2. For the proof of \eqref{estim3} we refer to \cite{G2}. 
	
	3. We proceed with the proof of \eqref{estim2} in the case $\sigma\in]0,1[$. The case of higher order regularity exponents may be done by  taking advantage of the explicit formula for partial derivatives of $\wt K_t$ that has been derived by M.~Grassin in \cite[p. 1404]{G2}. The same method as in the case $\sigma\in]0,1[$, which we shall shortly show, has then to be applied to each term of the formula. The details are left to the reader. 
	
	% Proving the result for higher order regularity exponents may be done by 
	% taking advantage of the explicit formula for partial derivatives of $\wt K_t$
	%  that has been derived by M.~Grassin
	% in \cite[p. 1404]{G2}. The same method as in the case $\sigma\in]0,1[$ 
	% has to be applied to each term of the formula. The details are left to the reader. 
	
	To bound $\wt K_t:=(1+t)^{-1} K(t,\cdot)$ in $\dot H^\sigma,$ 
	we use the following characterization of Sobolev norms by finite differences:
	$$
	\|\wt K_t\|_{\dot H^\sigma}^2 \approx \int_{\R^3}\!\!\int_{\R^3}
	\frac{|\wt K_t(y)-\wt K_t(x)|^2}{{|y-x|}^{3+2\sigma}}\,dx\,dy.
	$$
	We can thus write $ \wt K_t(y)-\wt K_t(x)= I^1_t(x,y)+I^2_t(x,y)$ where
	$$\begin{aligned}
		I_t^1(x,y)&= (\Id\!+\!tDv_0(X_t^{-1}(y)))^{-1}\bigl(Dv_0(X_t^{-1}(y))-Dv_0(X_t^{-1}(x))\bigr),\\
		I_t^2(x,y)&=t(\Id\!+\!tDv_0(X_t^{-1}(y)))^{-1}\!\bigl(Dv_0(X_t^{-1}(x))-
		\!Dv_0(X_t^{-1}(y))\bigr) \\ &\qquad \times (\Id\!+\!tDv_0(X_t^{-1}(x)))^{-1} (Dv_0(X_t^{-1}(x))-\Id).\end{aligned}$$ 
	We see, thanks to \eqref{eq:DXbounded} and to the
	change of variable $x'=X_t^{-1}(x)$ and $y'=X_t^{-1}(y),$ that  
	$$
	\int_{\R^3}\!\!\int_{\R^3}
	\frac{|I_t^1(x,y)|^2}{{|y-x|}^{3+2\sigma}}\,dx\,dy\leq\frac C{(1+\ep t)^2}
	\int_{\R^3}\!\!\int_{\R^3}\frac{|Dv_0(y')-Dv_0(x')|^2}{|X_t(y')-X_t(x')|^{3+2\sigma}}J_{X_t}(x') J_{X_t}(y') \,dx'\,dy'.
	$$
	Furthermore, we infer from \eqref{eq:DXbounded} that
	\begin{align}\label{lagrangegap}
		\begin{aligned}
			&|y'-x'|=|X_t^{-1}(X_t(y'))-X_t^{-1}(X_t(x'))|\\
			&\leq \|DX_t^{-1}\|_{L^\infty} |X_t(y')-X_t(x')| \leq \frac C{1+\ep t}  |X_t(y')-X_t(x')|.
		\end{aligned}
	\end{align}
	Therefore, using the fact that $\|J_{X_t}\|_{L^\infty}\leq C(1+\ep t)^3$ and \eqref{lagrangegap}, we get
	\begin{equation}\label{eq:It1}
		\int_{\R^3}\!\!\int_{\R^3}
		\frac{|I_t^1(x,y)|^2}{{|y-x|}^{3+2\sigma}}\,dx\,dy\leq  C(1+\ep t)^{1-2\sigma}
		\int_{\R^3}\!\!\int_{\R^3}
		\frac{|Dv_0(y')-Dv_0(x')|^2}{{|y'-x'|}^{3+2\sigma}}\,dx\,dy.\end{equation}
	Similarly, \eqref{eq:DXbounded} and the change of variable $x'=X_t^{-1}(x)$ and $y'=X_t^{-1}(y)$
	imply that 
	$$
	\int_{\R^3}\!\!\int_{\R^3}
	\frac{|I_t^2(x,y)|^2}{{|y-x|}^{d+2\sigma}}\,dx\,dy\leq\frac{Ct^2}{(1+\ep t)^4}
	\int_{\R^3}\!\!\int_{\R^3}\frac{|Dv_0(y')-Dv_0(x')|^2}{|X_t(y')-X_t(x')|^{3+2\sigma}}
	J_{X_t}(x') J_{X_t(y')} \,dx'\,dy',
	$$
	and we thus also have \eqref{eq:It1} for $I_t^2.$  As a conclusion, using
	the characterization of $\|Dv_0\|_{\dot H^\sigma}$ by finite difference, we get
	$$\|\wt K_t\|_{\dot H^\sigma}\leq C(1+\ep t)^{\frac 12-\sigma} \|Dv_0\|_{\dot H^\sigma},
	$$
	which gives the desired estimate for $\sigma\in]0,1[.$ 
	% \medbreak
	% Proving the result for higher order regularity exponents may be done by 
	% taking advantage of the explicit formula for partial derivatives of $\wt K_t$
	%  that has been derived by M.~Grassin
	% in \cite[p. 1404]{G2}. The same method as in the case $\sigma\in]0,1[$ 
	% has to be applied to each term of the formula. The details are left to the reader. 
	\qed

	%%%%%%%%%%%%%%%%%%%%%%%%%%%%%%%%%%%%%%%%%%%
	
	\subsection{Sobolev estimates for System {\eqref{eq:MP}}}\label{ss:sob}

	% The goal of the present section is to prove  a priori 
	%  estimates  for the discrepancy between  
	% the solution to  \eqref{eq:MP} and to \eqref{auxiso}. 
	
	Let  $(v,\overline E, \overline B),$ be the solution of the Burgers-Maxwell system given by Proposition \ref{p:Burgers}
	and Lemma \ref{EMfree}. 
	Consider  a sufficiently smooth solution  $(\rho,u,E,B)$ of \eqref{eq:MP} on $[0,T]\times\R^3,$ 
	and  set   $w:=u-v$, $e:=E-\overline E$ and $b:=B-\overline B.$ Then, $(\rho,w,e,b)$ satisfies:
	\begin{equation}
		\label{eq:BB}
		\left\{
		\begin{array}{l} 
			(\d_t+w\cdot\nabla)\rho+\frac{\gamma-1}2\rho\,\div w+v\cdot\nabla\rho+\frac{\gamma-1}2\rho\,\div v=0,\\[1.5ex]
			(\d_t+w\cdot\nabla)w+\frac{\gamma-1}2\rho\,\nabla\rho+v\cdot\nabla w+w\cdot\nabla v\\[1.5ex]
			=-E-u\times B+v\times\Curl v
			%-\eta w
			,\\[1.5ex]
			\d_t e-\Curl b=\left( \frac{\gamma-1}{2\sqrt{A\gamma}} \right)^{\frac{2}{\gamma-1}}\rho^{\frac2{\gamma-1}}u-\alpha_1 e,\\[1.5ex]
			\d_t b+\Curl e=-\alpha_2 b.
		\end{array}
		\right.
	\end{equation}
	Of course, we have the compatibility conditions 
	\begin{equation}
		\label{eq:compatibilty}
		\Div e=-\tilde\vr\andf \Div b=0.
	\end{equation}
	Our aim is to prove decay estimates in $\dot H^\sigma$ for \eqref{eq:BB}  for all $0\leq\sigma\leq s.$
	Clearly, arguing by interpolation, it suffices to consider the border cases $\sigma=0$ and $\sigma=s.$
	\medbreak
	1. Let us start with  $\sigma=0.$ Taking the $L^2$ 
	scalar product of the first equation of \eqref{eq:BB} with $\rho$ gives
	\begin{align}\label{eq:L20}
		\begin{aligned}
			\frac{1}{2}&\frac d{dt}\|\rho\|_{L^2}^2
			-\frac12\int_{\R^3}\rho^2\div w\,dx
			+\frac{\gamma-1}2\int_{\R^3}\rho^2\div w\,dx\\
			& -\frac12\int_{\R^3}\rho^2\div v\,dx + \frac{\gamma-1}2 \int_{\R^3} \rho^2 \div v\, dx=0.
		\end{aligned}
	\end{align}
	% $$\frac12\frac d{dt}\|\rho\|_{L^2}^2
	% -\frac12\int_{\R^3}\rho^2\div w\,dx
	% +\frac{\gamma-1}2\int_{\R^3}\rho^2\div w\,dx
	% -\frac12\int_{\R^3}\rho^2\div v\,dx 
	% $$
	% \begin{equation}
		% \label{eq:L20}
		% +\frac{\gamma-1}2 \int_{\R^3} \rho^2 \div v\, dx=0.
		% \end{equation}
	
	In order to bound the magnetic term, we first remark, as observed by Germain and Masmoudi in \cite{GM} that the quantity  $Z:=B-\Curl u$  (seen as a vector-field)
	is conserved by the flow of $u$.
	Therefore, according to  Hypothesis $({\bf H 4})$, we have 
	\begin{equation}
		\label{eq:Bw}
		B=\Curl u.
	\end{equation}
	Using this observation, we see that the Lorentz contribution in the right hand side of the momentum equation --- after taking $L^2$ product with $w$ --- reduces to
	$$-\int_{\R^3}( E+v\times \Curl w )\cdot w\ dx.$$
	
	Taking now the $L^2$ scalar product of the second equation of \eqref{eq:BB} with $w$ and using (\ref{eq:Bw}) gives
	$$\frac12\frac d{dt}\|w\|_{L^2}^2
	-\frac{\gamma-1}4\int_{\R^3}\rho^2\div w\,dx
	%+\eta\|w\|_{L^2}^2
	-\frac12\int_{\R^3}|w|^2\div w\,dx
	+\int_{\R^3}w\ Dv\ w\,dx
	$$
	$$
	=-\int_{\R^3} E\cdot w\ dx
	+\frac12\int_{\R^3}|w|^2\div w\,dx.
	-\int_{\R^3}w(v\times\Curl w )\ dx.
	$$
	Using the identity $\frac{1}{2}\nabla w^2=w\times\Curl w + w\cdot\nabla w$ we further get
	\begin{align}\label{eq:L20bis}
		\begin{aligned}
			\frac12\frac d{dt}&\|w\|_{L^2}^2 - \frac{\gamma-1}4\int_{\R^3}\rho^2\div w\,dx - \frac12\int_{\R^3}|w|^2\div w\,dx + \frac12\int_{\R^3}w\ Dv\ w\,dx\\
			&=-\int_{\R^3} E\cdot w\ dx - \frac12\int_{\R^3}|w|^2\div v\,dx.
		\end{aligned}
	\end{align}
	% $$\frac12\frac d{dt}\|w\|_{L^2}^2
	% -\frac{\gamma-1}4\int_{\R^3}\rho^2\div w\,dx
	% %+\eta\|w\|_{L^2}^2
	% -\frac12\int_{\R^3}|w|^2\div w\,dx
	% +\frac12\int_{\R^3}w\ Dv\ w\,dx
	% $$
	% \begin{equation}
		% \label{eq:L20bis}
		% =-\int_{\R^3} E\cdot w\ dx
		% -\frac12\int_{\R^3}|w|^2\div v\,dx.
		% \end{equation}
	In the same stroke, taking the $L^2$ scalar product of the last two equations of  \eqref{eq:BB}  with $(e,b)$ gives
	\begin{equation}
		\label{eq:L20ter}
		\frac12\frac d{dt}\left(\|e\|_{L^2}^2+\|b\|_{L^2}^2\right)+\alpha_1\|e\|_{L^2}^2+\alpha_2\|b\|_{L^2}^2 
		=\left( \frac{\gamma-1}{2\sqrt{A\gamma}} \right)^{\frac{2}{\gamma-1}}\int_{\R^3}\rho^{\frac2{\gamma-1}}e\cdot u\ dx.
	\end{equation}
	Let us compute several contributions in (\ref{eq:L20bis}) and (\ref{eq:L20ter}). From \eqref{Du} and \eqref{eq:divv}, one gets
	$$
	\int_{\R^3}w\ Dv\ w\,dx
	=\frac{1}{1+t}\ \| w\|_{L^2}^2
	+\frac{1}{(1+t)^2}\int_{\R^3} w(x)K(t,x) w(x)\ dx,$$
	and
	$$\int_{\R^3}|w|^2\div v\,dx
	=\frac{3}{1+t}\ \|w\|_{L^2}^2
	+\frac{1}{(1+t)^2}\int_{\R^3}TrK(t,x) |w|^2\ dx.$$
	Secondly, 
	$$
	\left|\int_{\R^3} E\cdot w\ dx\right|
	\leq
	\left(\|e\|_{L^2}+\|\overline{E}\|_{L^2}\right)\|w\|_{L^2}
	%+\frac{1}{2\eta}\|e\|_{L^2}^2
	,$$
	%where we used Lemmas \ref{EMfree} and \ref{Elecmag}, 
	and finally
	\begin{align*}
		\left|\int_{\R^3}\rho^{\frac2{\gamma-1}}e\cdot u\ dx\right| \lesssim  \|\rho^{\frac{2}{\gamma-1}}\|_{L^\infty}\|e\|_{L^2}\left( \|w\|_{L^2} + \|v\|_{L^2} \right)
	\end{align*}
	% $$\left|\int_{\R^3}\rho^{\frac2{\gamma-1}}e\cdot u\ dx\right|
	% \leq \left|\int_{\R^3}\rho^{\frac2{\gamma-1}-\frac{1}{2}}e\cdot \rho^{\frac{1}{2}} u\ dx\right|
	% \leq C\|\rho^{\frac4{\gamma-1}-1}\|_{L^\infty}.$$
	Let us set now
	\begin{equation}
		\label{eq:cdg}
		c_{\gamma}:= \min\left\{1,3\frac{\gamma-1}{2}\right\}-\frac{3}{2}\cdotp
	\end{equation}
	{}From \eqref{eq:L20}, Cauchy-Schwarz inequality and Proposition \ref{p:Burgers}, we deduce that, denoting by $M$ a bound of $K,$ 
	$$\frac12\frac d{dt}\|(\rho,w,e,b)\|_{L^2}^2
	+\frac{c_{\gamma}}{1+t}\|(\rho,w,e,b)\|_{L^2}^2$$
	$$
	\lesssim  \frac{M\max\left\{5/2,3/2|\gamma-2|\right\}}{(1+t)^2}\|(\rho,w,e,b)\|_{L^2}^2
	+\|\div w\|_{L^\infty}\|(\rho,w,e,b)\|_{L^2}^2
	$$
	\begin{equation}
		+C |\rho^{\frac2{\gamma-1}}\|_{L^\infty}\|e\|_{L^2}(\|w\|_{L^2} + \|v\|_{L^2})
		+\left(\|e\|_{L^2}+\|\overline{E}\|_{L^2}\right)\|w\|_{L^2}.
		\label{eq:L2}
	\end{equation}
	%provided that $\eta\alpha>1$.
	\medbreak
	2. The case  $\sigma>0.$ In the sequel, we will use freely the following estimates proved in \cite{BDD,bddn} (see related results in \cite{KP}\cite{KPV}), which we shall also refer to sometimes as Kato-Ponce estimates. 
	
	\begin{lem}
		\label{l:com1} 
		\begin{itemize}
			\item If $s>0,$ then we have:
			$$\|[v,\dot\Lambda^s]u\|_{L^2}\lesssim
			\|v\|_{\dot H^s}\|u\|_{L^\infty}+\|\nabla v\|_{L^\infty}\|u\|_{\dot H^{s-1}}.$$
			\item If $s>1,$ then we have:
			$$\|[v,\dot\Lambda^s]u-s\nabla v\cdot\dot\Lambda^{s-2}\nabla u\|_{L^2}\lesssim
			\|v\|_{\dot H^s}\|u\|_{L^\infty}+\|\nabla^2v\|_{L^\infty}\|u\|_{\dot H^{s-2}}.$$
		\end{itemize}
	\end{lem}
	\medbreak

	In order to prove Sobolev estimates, we introduce  the homogeneous fractional differential operator 
	$\dot\Lambda^s$  defined by  $\cF(\dot\Lambda^sf)(\xi):=|\xi|^s\cF f(\xi)$ and observe that  
	$\rho_s:=\dot\Lambda^s\rho,$ $w_s:=\dot\Lambda^sw$, $E_s:=\dot\Lambda^sE$, $e_s:=\dot\Lambda^se$ and $b_s:=\dot\Lambda^sb$ satisfy
	(with the usual summation convention over repeated indices)
	\begin{align}
		(\d_t+w\cdot\nabla)\rho_s+\frac{\gamma-1}2\rho\,\div &w_s+v\cdot\nabla\rho_s-s\d_jv^k\dot\Lambda^{-2}\d^2_{jk}\rho_s
		+\frac{\gamma-1}2\dot\Lambda^s(\rho\,\div v)\nonumber\\
		&=\dot R^1_s+\dot R^2_s+\dot R^3_s,\label{eq:BBs1}
	\end{align}
	\begin{align}
		(\d_t&+w\cdot\nabla)w_s+\frac{\gamma-1}2\rho\nabla\rho_s+v\cdot\nabla w_s-s\d_jv^k\dot\Lambda^{-2}\d^2_{jk}w_s+\dot\Lambda^s(w\cdot\nabla v),\nonumber\\
		&+E_s + w_s\times \Curl w + v_s\times\Curl w + w_s\times\Curl v =\dot R^4_s+\dot R^5_s+\dot R^6_s+\dot R^7_s,\label{eq:BBs2}
	\end{align}
	\begin{equation}
		\label{eq:BBs3}
		\d_t e_s-\Curl b_s+\alpha_1 e_s=\dot R^8_s,
		%\dot\Lambda^s(\rho^{\frac2{\gamma-1}}u),
	\end{equation}
	\begin{equation}
		\label{eq:BBs4}
		\d_t b_s+\Curl e_s+\alpha_2 b_s=0,
	\end{equation}
	with 
	$$\begin{array}{ll}
		\dot R^1_s:=[w,\dot\Lambda^s]\nabla\rho,\quad & \dot R^5_s:=\frac{\gamma-1}2[\rho,\dot\Lambda^s]\nabla\rho,\\[1ex]
		\dot R^2_s:= \frac{\gamma-1}2[\rho,\dot\Lambda^s]\div w,\quad&\dot R^6_s:=[v,\dot\Lambda^s]\nabla w-s\d_jv^k\dot\Lambda^{-2}\d^2_{jk}w_s,\\[1ex]
		%\dot R_3^s:=\frac{\gamma-1}2[\div v,\dot\Lambda^s]\rho,\quad&\dot R_7^s:=\frac{\gamma-1}2[\nabla v,\dot\Lambda^s]w,\\[1ex]
		\dot R^3_s:=[v,\dot\Lambda^s]\nabla\rho-s\d_jv^k\dot\Lambda^{-2}\d^2_{jk}\rho_s,\quad&\dot R^7_s:=\dot\Lambda^s(v\times \Curl v-u\times \Curl u) 
		\\[1ex]
		&\qquad\quad - v_s\times\Curl v + u_s\times \Curl u,\\[1ex]
		\dot R^4_s:=[w,\dot\Lambda^s]\nabla w,\quad&
		\dot R^8_s:=\left( \frac{\gamma-1}{2\sqrt{A\gamma}} \right)^{\frac{2}{\gamma-1}}\dot\Lambda^s(\rho^{\frac2{\gamma-1}}u).
	\end{array}
	$$
	%and
	%$$\dot R^7_s:=[u\times,\dot\Lambda^s]\Curl u.$$
	The definitions of $\dot R^3_s$ and $\dot R^6_s$ are motivated
	by the fact that, according to  the classical  theory of pseudo-differential operators, we  expect to have 
	$$
	[\dot \Lambda^s,v]\cdot\nabla z=\frac1i\bigl\{|\xi|^s,v(x)\bigr\}(D)\nabla z +\hbox{remainder}.
	$$
	Computing  the \emph{Poisson bracket}  in the right-hand side  yields  
	$$
	\frac1i\bigl\{|\xi|^s,v(x)\bigr\}(D) = -s \d_j v \dot\Lambda^{s-2} \d_j.
	$$

	Now, taking advantage of \eqref{eq:Dv}, we get 
	$$-\d_jv^k\dot\Lambda^{-2}\d^2_{jk}z=\frac1{1+t}z-\frac{K_{kj}}{(1+t)^2}\dot\Lambda^{-2}\d^2_{jk}z,
	$$
	and using \eqref{eq:divv} yields
	$$\begin{aligned}
		\dot\Lambda^s(\rho\,\div v)&=\frac d{1+t}\rho_s+\frac1{(1+t)^2}\dot\Lambda^s(\rho\,{\rm Tr}\, K)\\
		\andf \dot \Lambda^s(w\cdot\nabla v)&=\frac1{1+t}w_s+\frac1{(1+t)^2}\dot\Lambda^s(K\cdot w).\end{aligned}$$

	Hence, taking the $L^2$ inner product of \eqref{eq:BBs1} \eqref{eq:BBs2} \eqref{eq:BBs3} and \eqref{eq:BBs4} respectively with $(\rho_s,w_s),$ denoting 
	$c_{\gamma,s}:= c_{\gamma}+s$, and using the fact that  $\|E_s\cdot w_s\|_{L^1}
	\leq \left(\|e_s\|_{L^2}+\|\overline{E}_s\|_{L^2}\right)\|w_s\|_{L^2},$ we end up with  
	\begin{align}\label{eq:Hs}
		\begin{aligned}
			&\frac12\frac d{dt}\|(\rho_s,w_s,e_s,b_s)\|_{L^2}^2+\frac{c_{\gamma,s}}{1+t}\|(\rho_s,w_s,e_s,b_s)\|_{L^2}^2 + \alpha(\|e_s\|_{L^2}^2+\|b_s\|_{L^2}^2) \\
			&\lesssim \frac{\gamma-1}2\|\nabla \rho\|_{L^\infty}\|\rho_s\|_{L^2}\|w_s\|_{L^2} + \frac{sM}{(1+t)^2} \|(\rho_s,w_s,e_s,b_s)\|_{L^2}^2\\
			&\quad +\frac{1}{(1+t)^2}\left(\frac{\gamma-1}2\|\dot\Lambda^s(\rho{\rm Tr}\, K)\|_{L^2}\|\rho_s\|_{L^2} + \|\dot\Lambda^s(K\cdot w)\|_{L^2}\|w_s\|_{L^2}\right)\\
			&\quad +\left(\|e_s\|_{L^2}+\|\overline{E}_s\|_{L^2} + \|v\|_{\dot H^s}\|\Curl w\|_{L^\infty}\right)\|w_s\|_{L^2} +  \sum_{j=1}^3\|\dot R_s^j\|_{L^2}\|\rho_s\|_{L^2}\\
			&\quad + \sum_{j=4}^7\|\dot R_s^j\|_{L^2}\|w_s\|_{L^2} + \|\dot R_s^8\|_{L^2}\|e_s\|_{L^2}.
		\end{aligned}
	\end{align}

	The  terms $\dot R_s^1,$ $\dot R_s^2,$ $\dot R_s^4,$ $\dot R_s^5$ and $\dot R_s^7$  
	may be treated according to Lemma \ref{l:com1} which gives us
	% $$\|\dot R^1_s\|_{L^2}\lesssim\|\dot\Lambda^1\rho\|_{L^\infty}\|\dot\Lambda^{s-1}w\|_{L^2}
	% +\| w\dot\Lambda^{s}\nabla\rho\|_{L^2}
	% +\|\nabla \rho\dot\Lambda^{s}w\|_{L^2}
	% +\|\nabla w\dot\Lambda^{s-1}\nabla\rho\|_{L^2},$$
	% then
	$$\begin{aligned}
		\|\dot R^1_s\|_{L^2}&\lesssim \|\nabla\rho\|_{L^\infty}\|\nabla w\|_{\dot H^{s-1}}+\|\nabla w\|_{L^\infty}\|\rho\|_{\dot H^{s}},\\
		\|\dot R_s^4\|_{L^2}&\lesssim \|\nabla w\|_{L^\infty}\|w\|_{\dot H^{s}},\\
		\|\dot R_s^2\|_{L^2}&\lesssim \|\div w\|_{L^\infty}\|\rho\|_{\dot H^{s}}+\|\nabla \rho\|_{L^\infty}\|\div w\|_{\dot H^{s-1}},\\
		\|\dot R_s^5\|_{L^2}&\lesssim \|\nabla\rho\|_{L^\infty}\|\rho\|_{\dot H^{s}},\\
		%\|\dot R_s^7\|_{L^2}&\lesssim \|\nabla u\|_{L^\infty}\|\nabla u\|_{\dot H^{s-1}}.
	\end{aligned}$$
	The (more involved) terms $\dot R_s^3$ and $\dot R_s^6$  may be also handled thanks to Lemma~\ref{l:com1}.
	We get
	$$\begin{aligned}
		\|\dot R_s^3\|_{L^2}
		&\lesssim \|\nabla\rho\|_{L^\infty}\|v\|_{\dot H^{s}}+\|\nabla^2v\|_{L^\infty}\|\nabla\rho\|_{\dot H^{s-2}},\\
		\|\dot R_s^6\|_{L^2}&\lesssim \|\nabla w\|_{L^\infty}\|v\|_{\dot H^{s}}+\|\nabla^2v\|_{L^\infty}\|\nabla w\|_{\dot H^{s-2}}.
	\end{aligned}$$
	We see also that
	% \begin{align*}
		%     \begin{aligned}
			%         \|\dot R_s^7\|_{L^2}\lesssim&\, \|\dot\Lambda^s(w\times \Curl w)- w\times \dot\Lambda^s(\Curl w)\|_{L^2} + \|\dot\Lambda^s(w\times \Curl v)- w\times \dot\Lambda^s(\Curl v)\|_{L^2}\\
			%         &+ \|\dot\Lambda^s(v\times \Curl w)- v\times \dot\Lambda^s(\Curl w)\|_{L^2} + \| w\times \dot\Lambda^s(\Curl w)\|_{L^2}\\
			%         & + \| w\times \dot\Lambda^s(\Curl v)\|_{L^2} + \| v\times \dot\Lambda^s(\Curl w)\|_{L^2}
			%     \end{aligned}
		% \end{align*}
	$$\|\dot R_s^7\|_{L^2}\lesssim 
	\|\dot\Lambda^s( \Curl w\times w)- \Curl w\times w_s\|_{L^2}
	+\|\dot\Lambda^s( \Curl v\times w)- \Curl v\times w_s\|_{L^2}
	$$
	$$
	+\|\dot\Lambda^s( \Curl w\times v)- \Curl w\times v_s\|_{L^2}.
	$$
	Therefore, Lemma \ref{l:com1} gives us 
	% \begin{align*}
		%     \begin{aligned}
			%         \|\dot R_s^7\|_{L^2} \lesssim & \,\|w\|_{\dot H^{s}}\|\Curl w\|_{L^\infty} + \|\nabla w\|_{L^\infty}\|\Curl w\|_{\dot H^{s-1}} + \|w\|_{\dot H^{s}}\|\Curl v\|_{L^\infty}\\
			%         & + \|\nabla w\|_{L^\infty}\|\Curl v\|_{\dot H^{s-1}} + \|v\|_{\dot H^{s}}\|\Curl w\|_{L^\infty} + \|\nabla v\|_{L^\infty}\|\Curl w\|_{\dot H^{s-1}}\\
			%         & + \|w\|_{L^\infty}\|\Curl w\|_{\dot H^{s}} + \|w\|_{L^\infty}\|\Curl v\|_{\dot H^{s}} + \|v\|_{L^\infty}\|\Curl w\|_{\dot H^{s}}
			%     \end{aligned}
		% \end{align*}
	\begin{align*}
		\|\dot R_s^7\|_{L^2}
		\lesssim 
		\|\nabla v\|_{L^\infty}\|\Curl w\|_{\dot H^{s-1}}
		+\|\Curl w\|_{L^\infty}\|v\|_{\dot H^{s}}+\|\nabla w\|_{L^\infty}\|\Curl v\|_{\dot H^{s-1}}\\
		+\|\Curl v\|_{L^\infty}\|w\|_{\dot H^{s}}+\|\nabla w\|_{L^\infty}\|\Curl w\|_{\dot H^{s-1}}+\|\Curl w\|_{L^\infty}\|w\|_{\dot H^{s}}.
	\end{align*}
	% $$
	% \|\dot R_s^7\|_{L^2}
	% \lesssim 
	% \|\nabla v\|_{L^\infty}\|\Curl w\|_{\dot H^{s-1}}
	% +\|\Curl w\|_{L^\infty}\|v\|_{\dot H^{s}}+\|\nabla w\|_{L^\infty}\|\Curl v\|_{\dot H^{s-1}}$$
	% $$+\|\Curl v\|_{L^\infty}\|w\|_{\dot H^{s}}+\|\nabla w\|_{L^\infty}\|\Curl w\|_{\dot H^{s-1}}+\|\Curl w\|_{L^\infty}\|w\|_{\dot H^{s}}.$$
	
	To bound $\dot R^8_s$ we first observe that --- due to \cite[Lemma A.2]{bddn} ---
	\begin{align}
		\label{eq:Inegs}
		\left\|\rho^{\frac{2}{\gamma-1}}\right\|_{\dot H^{\sigma}}
		\lesssim 
		\left\|\rho\right\|_{L^\infty}^{\frac{2}{\gamma-1}-1}\|\rho\|_{\dot H^{\sigma}} \quad\text{ for }\quad 0< \sigma < \frac{1}{2} + \frac{2}{\gamma-1}.
	\end{align} 
	% \textcolor{red}{\footnote{The term on the left-hand side is of the form $\|[\rho^{\frac{2}{\gamma-1}},\dot\Lambda^s]v \|_{L^2}$ but we used Kato-Ponce to $\|[v,\dot\Lambda^s]\rho^{\frac{2}{\gamma-1}} \|_{L^2}$. The box below shows a recomputation for estimating $\|\dot\Lambda^s(\rho^{\frac{2}{\gamma-1}}v)\|_{L^2}$.}Using Kato-Ponce estimate we get now
		% $$
		% \left\|\dot\Lambda^s\left( \rho^{\frac{2}{\gamma-1}}v\right)-\rho^{\frac{2}{\gamma-1}}v_s\right\|_{L^2}
		% \lesssim 
		% \|\nabla v\|_{L^\infty}
		%  \left\| \rho^{\frac{2}{\gamma-1}}\right\|_{\dot H^{s-1}}
		% +\left\|\rho^{\frac{2}{\gamma-1}}\right\|_{L^\infty} \|v\|_{\dot H^{s}}.
		% $$
		% Therefore
		% $$\|\dot \Lambda^s(\rho^{\frac{2}{\gamma-1}}v)\|_{L^2}
		% \lesssim 
		% \left\|\dot\Lambda^s\left( \rho^{\frac{2}{\gamma-1}}v\right)-\rho^{\frac{2}{\gamma-1}}v_s\right\|_{L^2}
		% +\left\|\rho^{\frac{2}{\gamma-1}}v_s\right\|_{L^2}
		% $$
		% $$
		% \lesssim 
		% \|\nabla v\|_{L^\infty}
		%  \left\| \rho^{\frac{2}{\gamma-1}}\right\|_{\dot H^{s-1}}
		% +\left\|\rho^{\frac{2}{\gamma-1}}\right\|_{L^\infty} \|v\|_{\dot H^{s}}.
		% $$
		% }
	Using Kato-Ponce estimate we get now
	$$
	\left\|[v,\dot\Lambda^s]\rho^{\frac{2}{\gamma-1}} \right\|_{L^2}
	\lesssim 
	\|\nabla v\|_{L^\infty}
	\left\| \rho^{\frac{2}{\gamma-1}}\right\|_{\dot H^{s-1}}
	+\left\|\rho^{\frac{2}{\gamma-1}}\right\|_{L^\infty} \|v\|_{\dot H^{s}}.
	$$
	Therefore
	\begin{align*}
		&\left\|\dot \Lambda^s\left(\rho^{\frac{2}{\gamma-1}}v\right)\right\|_{L^2}
		\lesssim \left\|[v,\dot\Lambda^s]\rho^{\frac{2}{\gamma-1}} \right\|_{L^2} + \left\|v \dot\Lambda^s\left( \rho^{\frac{2}{\gamma-1}}\right)\right\|_{L^2}\\
		&\lesssim 
		\|\nabla v\|_{L^\infty}
		\left\| \rho^{\frac{2}{\gamma-1}}\right\|_{\dot H^{s-1}}
		+ \left\|\rho^{\frac{2}{\gamma-1}}\right\|_{L^\infty} \|v\|_{\dot H^{s}}
		+\|v\|_{L^\infty} \left\| \rho^{\frac{2}{\gamma-1}}\right\|_{\dot H^{s}}
	\end{align*}
	% $$\left\|\dot \Lambda^s\left(\rho^{\frac{2}{\gamma-1}}v\right)\right\|_{L^2}
	% \lesssim 
	% 	\left\|[v,\dot\Lambda^s]\rho^{\frac{2}{\gamma-1}} \right\|_{L^2} + \left\|v \dot\Lambda^s\left( \rho^{\frac{2}{\gamma-1}}\right)\right\|_{L^2}
	% $$
	% $$
	% \lesssim 
	% \|\nabla v\|_{L^\infty}
	% \left\| \rho^{\frac{2}{\gamma-1}}\right\|_{\dot H^{s-1}}
	% + \left\|\rho^{\frac{2}{\gamma-1}}\right\|_{L^\infty} \|v\|_{\dot H^{s}}
	% +\|v\|_{L^\infty} \left\| \rho^{\frac{2}{\gamma-1}}\right\|_{\dot H^{s}}
	% $$

	% \textcolor{red}{\footnote{Arguing as in the box above these computations are altered (see box below)}Analogously
		% $$\|\dot \Lambda^s(\rho^{\frac{2}{\gamma-1}}w)\|_{L^2}
		% \lesssim 
		% \left\|\dot\Lambda^s\left( \rho^{\frac{2}{\gamma-1}}w\right)-\rho^{\frac{2}{\gamma-1}}w_s\right\|_{L^2}
		% +\left\|\rho^{\frac{2}{\gamma-1}}w_s\right\|_{L^2}
		% $$
		% $$
		% \lesssim 
		% \|\nabla w\|_{L^\infty}
		%  \left\| \rho^{\frac{2}{\gamma-1}}\right\|_{\dot H^{s-1}}
		% +\left\|\rho^{\frac{2}{\gamma-1}}\right\|_{L^\infty} \|w\|_{\dot H^{s}}.
		% $$
		% Therefore
		% $$\|\dot R^8_s\|_{L^2}
		% \lesssim 
		% \|\nabla v\|_{L^\infty}\left(\|\rho\|_{L^\infty}\right)^{\frac{2}{\gamma-1}-1} \|\rho\|_{\dot H^{s-1}}
		% +\left(\|\rho\|_{L^\infty}\right)^{\frac{2}{\gamma-1}} \|v\|_{\dot H^{s}}
		% $$
		% $$
		% +\|\nabla w\|_{L^\infty}\left(\|\rho\|_{L^\infty}\right)^{\frac{2}{\gamma-1}-1} \|\rho\|_{\dot H^{s-1}}
		% +\left(\|\rho\|_{L^\infty}\right)^{\frac{2}{\gamma-1}} \|w\|_{\dot H^{s}}.
		% $$
		% }
	
	Analogously
	\begin{align*}
		&\left\|\dot \Lambda^s\left(\rho^{\frac{2}{\gamma-1}}w\right)\right\|_{L^2}
		\lesssim 
		\left\|[w,\dot\Lambda^s]\rho^{\frac{2}{\gamma-1}} \right\|_{L^2} + \left\|w \dot\Lambda^s\left( \rho^{\frac{2}{\gamma-1}}\right)\right\|_{L^2}\\
		&\lesssim 
		\|\nabla w\|_{L^\infty}
		\left\| \rho^{\frac{2}{\gamma-1}}\right\|_{\dot H^{s-1}}
		+ \left\|\rho^{\frac{2}{\gamma-1}}\right\|_{L^\infty} \|w\|_{\dot H^{s}}
		+\|w\|_{L^\infty} \left\| \rho^{\frac{2}{\gamma-1}}\right\|_{\dot H^{s}}.
	\end{align*}
	% $$\left\|\dot \Lambda^s\left(\rho^{\frac{2}{\gamma-1}}w\right)\right\|_{L^2}
	% \lesssim 
	% \left\|[w,\dot\Lambda^s]\rho^{\frac{2}{\gamma-1}} \right\|_{L^2} + \left\|w \dot\Lambda^s\left( \rho^{\frac{2}{\gamma-1}}\right)\right\|_{L^2}
	% $$
	% $$
	% \lesssim 
	% \|\nabla w\|_{L^\infty}
	% \left\| \rho^{\frac{2}{\gamma-1}}\right\|_{\dot H^{s-1}}
	% + \left\|\rho^{\frac{2}{\gamma-1}}\right\|_{L^\infty} \|w\|_{\dot H^{s}}
	% +\|w\|_{L^\infty} \left\| \rho^{\frac{2}{\gamma-1}}\right\|_{\dot H^{s}}
	% $$
	Therefore, with reinforcement of \eqref{eq:Inegs}, we have
	\begin{align*}
		&\|\dot R^8_s\|_{L^2}
		\lesssim 
		\|\nabla v\|_{L^\infty}
		\left\| \rho\right\|_{L^\infty}^{\frac{2}{\gamma-1}-1}\left\| \rho\right\|_{\dot H^{s-1}}
		+ \left\|\rho\right\|_{L^\infty}^{\frac{2}{\gamma-1}} \|v\|_{\dot H^{s}}
		+\|v\|_{L^\infty} \left\| \rho\right\|_{L^\infty}^{\frac{2}{\gamma-1}-1}\left\| \rho\right\|_{\dot H^{s}}\\
		&+\|\nabla w\|_{L^\infty}
		\left\| \rho\right\|_{L^\infty}^{\frac{2}{\gamma-1}-1}\left\| \rho\right\|_{\dot H^{s-1}}
		+ \left\|\rho\right\|_{L^\infty}^{\frac{2}{\gamma-1}} \|w\|_{\dot H^{s}}
		+\|w\|_{L^\infty} \left\| \rho\right\|_{L^\infty}^{\frac{2}{\gamma-1}-1}\left\| \rho\right\|_{\dot H^{s}}.
	\end{align*}
	% $$\|\dot R^8_s\|_{L^2}
	% \lesssim 
	% \|\nabla v\|_{L^\infty}
	% \left\| \rho\right\|_{L^\infty}^{\frac{2}{\gamma-1}-1}\left\| \rho\right\|_{\dot H^{s-1}}
	% + \left\|\rho\right\|_{L^\infty}^{\frac{2}{\gamma-1}} \|v\|_{\dot H^{s}}
	% +\|v\|_{L^\infty} \left\| \rho\right\|_{L^\infty}^{\frac{2}{\gamma-1}-1}\left\| \rho\right\|_{\dot H^{s}}
	% $$
	% $$
	% +\|\nabla w\|_{L^\infty}
	% \left\| \rho\right\|_{L^\infty}^{\frac{2}{\gamma-1}-1}\left\| \rho\right\|_{\dot H^{s-1}}
	% + \left\|\rho\right\|_{L^\infty}^{\frac{2}{\gamma-1}} \|w\|_{\dot H^{s}}
	% +\|w\|_{L^\infty} \left\| \rho\right\|_{L^\infty}^{\frac{2}{\gamma-1}-1}\left\| \rho\right\|_{\dot H^{s}}.
	% $$

	% \textcolor{red}{\footnote{I don't seem to understand what we meant when we used \textit{tame estimates} here. Nevertheless, I tried using the same technique as above, see box below.}Finally, the standard Sobolev tame estimate yields 
		% $$\begin{aligned}
			% \|\dot\Lambda^s(\rho{\rm Tr}\, K)\|_{L^2}
			% &\lesssim\|\rho\|_{L^\infty}\|\nabla({\rm Tr}\, K)\|_{\dot H^{s-1}}+\|{\rm Tr}\, K\|_{L^\infty}\|\rho\|_{\dot H^s}\\
			% \andf \|\dot\Lambda^s(K\cdot w)\|_{L^2}&\lesssim\|w\|_{L^\infty}\|\nabla K\|_{\dot H^{s-1}}+\|K\|_{L^\infty}\|w\|_{\dot H^s}.
			% \end{aligned}
		% $$}
	
	Using similar arguments as above we get the following estimes
	$$\begin{aligned}
		\|\dot\Lambda^s(\rho{\rm Tr}\, K)\|_{L^2}
		&\lesssim\|\nabla\rho\|_{L^\infty}\|{\rm Tr}\, K\|_{\dot H^{s-1}}+\|{\rm Tr}\, K\|_{L^\infty}\|\rho\|_{\dot H^s} + \|\rho\|_{L^\infty}\|{\rm Tr}\, K\|_{\dot H^{s}}\\
		\andf \|\dot\Lambda^s(K\cdot w)\|_{L^2}&\lesssim\|\nabla w\|_{L^\infty}\|K\|_{\dot H^{s-1}}+\|K\|_{L^\infty}\|w\|_{\dot H^s} + \|w\|_{L^\infty}\|K\|_{\dot H^{s}}.
	\end{aligned}
	$$

	Plugging all the above estimates in \eqref{eq:Hs} and using Proposition \ref{p:Burgers} 
	%and Lemma \ref{EMfree} and \ref{Elecmag}
	, we end up with 
	\begin{align}\label{eq:Hs2}
		\begin{aligned}
			& \frac12\frac d{dt}\|(\rho,w,e,b)\|_{\dot H^s}^2 + \frac{c_{\gamma,s}}{1+t}\|(\rho,w,e,b)\|_{\dot H^s}^2\\
			& \lesssim \frac{1}{(1+t)^2}\|(\rho,w,e,b)\|_{\dot H^s}^2 +(\|e_s\|_{L^2}+\|\overline{E}_s\|_{L^2})\|(\rho,w,e,b)\|_{\dot H^s}\\
			& + \|(\nabla\rho,\nabla w)\|_{L^\infty}\|(\rho,w,e,b)\|_{\dot H^s}^2 + \frac{1}{(1+t)^{s+\frac{1}{2}}}\|(\nabla\rho,\nabla w)\|_{L^\infty}\|(\rho,w,e,b)\|_{\dot H^s}\\
			& + \frac{1}{(1+t)}\|(\rho,w,e,b)\|_{\dot H^s}^2 + \frac{1}{(1+t)^{s+\frac{3}{2}}}\|(\rho,w)\|_{L^\infty}\|(\rho,w,e,b)\|_{\dot H^s}\\
			& + \frac{1}{(1+t)^{s-\frac{1}{2}}}\|(\nabla\rho,\nabla w)\|_{L^\infty}\|(\rho,w,e,b)\|_{\dot H^s} + \frac{1}{(1+t)^3}\|(\rho,w)\|_{\dot H^{s-1}}\|(\rho,w,e,b)\|_{\dot H^s}\\
			& + \|\rho\|_{L^\infty}^{\frac{2}{\gamma - 1} - 1}\|(\rho,w,e,b)\|_{\dot H^s}^2 + \frac{1}{(1+t)^{s-\frac{1}{2}}}\|\rho\|_{L^\infty}^{\frac{2}{\gamma - 1} }\|(\rho,w,e,b)\|_{\dot H^s} \\
			&  + \frac{1}{(1+t)}\|\rho\|_{L^\infty}^{\frac{2}{\gamma - 1} -1 }\|\rho\|_{\dot H^s}^2 + \|\nabla w\|_{L^\infty}\|\rho\|_{L^\infty}^{\frac{2}{\gamma-1}-1}\|\rho\|_{\dot H^{s-1}}\|\rho\|_{\dot H^s}\\
			& + \|\rho\|_{L^\infty}^{\frac{2}{\gamma - 1}}\|\rho\|_{\dot H^s}\|w\|_{\dot H^s}  + \|w\|_{L^\infty} \|\rho\|_{L^\infty}^{\frac{2}{\gamma - 1}-1}\|\rho\|_{\dot H^s}^2
		\end{aligned}
	\end{align}

	Let us introduce the  notation
	$$\dot X_\sigma:=\|(\rho,w,e,b)\|_{\dot H^\sigma}\andf
	X_\sigma:=\sqrt{\dot X_0^2+\dot X_\sigma^2}\approx \|(\rho,w,e,b)\|_{H^\sigma}
	\quad\hbox{for }\ \sigma\geq0.$$
	Our aim is to bound the right-hand side of \eqref{eq:L2} and  \eqref{eq:Hs2} in terms of $\dot X_0$ and $\dot X_s$ only. 
	\vskip0.5cm
	Arguing by interpolation, we get first:
	\begin{eqnarray}\label{eq:interpo1}
		\|(\rho,w)\|_{L^\infty}&\!\!\!\lesssim\!\!\!& \dot X_0^{1-\frac 3{2s}}\dot X_s^{\frac 3{2s}},\\\label{eq:interpo2}
		\|(D\rho,Dw)\|_{L^\infty}&\!\!\!\lesssim\!\!\!& \dot X_0^{1-\frac5{2s}}\dot X_s^{\frac 5{2s}},\\\label{eq:interpo3}
		\|(\rho,w)\|_{\dot H^{s-1}}&\!\!\!\lesssim\!\!\!& \dot X_0^{\frac1s}\dot X_s^{1-\frac1s}.
	\end{eqnarray}
	Then, plugging these inequalities and those of Proposition \ref{p:Burgers} in \eqref{eq:L2} and \eqref{eq:Hs2}, together with decay properties given by Lemma \ref{EMfree} and \ref{Elecmag} yields

	$$
	\frac{d}{dt}\dot X_0
	+\frac{c_{\gamma}}{1+t}\dot X_0
	\lesssim 
	\frac{\dot X_0}{(1+t)^2}
	+\dot X_0^{2-\frac{5}{2s}} \dot X_s^{\frac{5}{2s}}
	+
	\dot X_0^{(1-\frac{3}{2s})\frac{2}{\gamma-1} + 1} 
	\dot X_s^{\frac{3}{2s}\frac{2}{\gamma-1}},
	$$

	and
	
	\begin{align*}
		\begin{aligned}
			& \frac{d}{dt}\dot X_s + \frac{c_{\gamma,s}}{1+t}\dot X_s \lesssim \frac{\dot X_s}{(1+t)^2} + \dot X_0^{1-\frac{5}{2s}} \dot X_s^\frac5{2s}\dot X_s + \frac{\dot X_0^{1-\frac{5}{2s}}\dot X_s^{\frac{5}{2s}}}{(1+t)^{s+\frac{1}{2}}} + \frac{\dot X_s}{(1+t)}\\
			& + \frac{\dot X_0^{1-\frac{3}{2s}}\dot X_s^{\frac{3}{2s}}}{(1+t)^{s+\frac{3}{2}}} + \frac{\dot X_0^{1-\frac{5}{2s}}\dot X_s^{\frac{5}{2s}}}{(1+t)^{s-\frac{1}{2}}} + \frac{\dot X_0^{\frac{1}{s}}\dot X_s^{1-\frac{1}{s}}}{(1+t)^3} + \left(\dot X_0^{1-\frac{3}{2s}}\dot X_s^{\frac{3}{2s}} \right)^{\frac{2}{\gamma - 1}-1}\dot X_s\\
			& + \frac{\big(\dot X_0^{1-\frac{3}{2s}}\dot X_s^{\frac{3}{2s}} \big)^{\frac{2}{\gamma - 1}}}{(1+t)^{s-\frac{1}{2}}} + \frac{\big( \dot X_0^{1-\frac{3}{2s}}\dot X_s^{\frac{3}{2s}} \big)^{\frac{2}{\gamma - 1} - 1}\dot X_s}{(1+t)} + \dot X_0^{1-\frac{3}{2s}}\dot X_s^{1+\frac{3}{2s}}\big( \dot X_0^{1-\frac{3}{2s}}\dot X_s^{\frac{3}{2s}} \big)^{\frac{2}{\gamma - 1} - 1} .
		\end{aligned}
	\end{align*}
	
	% $$
	% \frac d{dt}\dot X_s
	% +\frac{c_{\gamma,s}}{1+t}\dot X_s
	% \lesssim
	% \frac{\dot X_s}{(1+t)^2}
	% +\frac{\dot X_0^{1-\frac 3{2s}}\dot X_s^{\frac 3{2s}}}{(1+t)^{s+1/2}}
	% +\dot X_0^{1-\frac{5}{2s}} \dot X_s^\frac5{2s}\dot X_s
	% +\frac{\dot X_0^{\frac1s}\dot X_s^{1-\frac1s}}{(1+t)^3}
	% $$
	% $$
	% +\frac{\dot X_s}{1+t}
	% +\frac{1}{(1+t)^{s-1/2}}\dot X_0^{1-\frac{5}{2s}} \dot X_s^\frac5{2s}
	% +\frac{1}{(1+t)^{s-1/2}}\dot X_0^{1-\frac{5}{2s}} \dot X_s^\frac5{2s}
	% +\frac{\dot X_s}{1+t}
	% +\dot X_0^{1-\frac{5}{2s}} \dot X_s^\frac5{2s}\dot X_s
	% $$
	% $$
	% +\dot X_0^{1-\frac{5}{2s}} \dot X_s^\frac5{2s}\dot X_s
	% +\frac{1}{(1+t)^{s-1/2}}\dot X_0^{1-\frac{5}{2s}} \dot X_s^\frac5{2s}
	% $$
	% $$
	% +\frac{1}{1+t}\left(\dot X_0^{1-\frac 3{2s}}\dot X_s^{\frac 3{2s}} \right)^{\frac{2}{\gamma-1}-1}\dot X_0^{\frac{1}{s}} \dot X_s^{1-\frac{1}{s}}
	% +\frac{1}{(1+t)^{s-3/2}}\left(\dot X_0^{1-\frac 3{2s}}\dot X_s^{\frac 3{2s}} \right)^{\frac{2}{\gamma-1}}
	% $$
	% $$
	% +\dot X_0^{1-\frac{5}{2s}} \dot X_s^\frac5{2s}\times\left(\dot X_0^{1-\frac 3{2s}}\dot X_s^{\frac 3{2s}} \right)^{\frac{2}{\gamma-1}-1}\dot X_0^{\frac{1}{s}} \dot X_s^{1-\frac{1}{s}}
	% +\left(\dot X_0^{1-\frac 3{2s}}\dot X_s^{\frac 3{2s}} \right)^{\frac{2}{\gamma-1}}\dot X_s
	% \cdotp
	% $$

	Let $a$ be an arbitrary positive number. Performing the change of unknown
	$$\dot Y_s=(1+t)^{c_{\gamma,s}-a}\dot X_s,$$ 
	%where $\displaystyle{\phi_{s,a}(t):=e^{(a-c_{\gamma,s})\frac{1-e^{-\eta t}}{\eta}}}$
	we observe that
	$$
	\frac d{dt}\dot Y_s+\frac{a}{1+t}\dot Y_s
	=(1+t)^{c_{\gamma,s}-a}\left(\frac d{dt}\dot X_s+\frac{c_{\gamma,s}}{1+t}\dot X_s\right).
	$$
	
	Therefore introducing the notation $Y_\sigma:=\sqrt{\dot Y_0^2+\dot Y_\sigma^2}$ for any $\sigma\in[0,s]$, we get
	
	$$
	\frac d{dt} Y_s+\frac{a}{1+t} Y_s
	\leq
	C\left(
	\frac{ Y_s}{(1+t)^2}
	+\frac{ Y_s^2}{(1+t)^{\Gamma_0}}
	+\frac{Y_s^{\frac{2}{\gamma-1}-1}}{(1+t)^{\Gamma_1}}
	+\frac{Y_s^{\frac{2}{\gamma-1}}}{(1+t)^{\Gamma_2}}
	\right.$$
	\begin{equation}
		\label{eq:Ineg}+\left. 
		\frac{Y_s^{\frac{2}{\gamma-1}-1}}{(1+t)^{\Gamma_3}}
		+\frac{Y_s^{\frac{2}{\gamma-1}+1}}{(1+t)^{\Gamma_4}}
		\right),
	\end{equation}
	with
	$\Gamma_0=c_{\gamma}-a+5/2 $,
	$\Gamma_1=(c_{\gamma}-a+3/2)(\frac{2}{\gamma-1}-1) $,
	$\Gamma_2= (c_{\gamma}-a+3/2)(\frac{2}{\gamma-1}-1)$,
	$\Gamma_3= (c_{\gamma}-a+3/2)(\frac{2}{\gamma-1})$ and
	$\Gamma_4= (c_{\gamma}-a+3/2)(\frac{2}{\gamma-1}-1)$.
	
	Denoting the new unknown $Z(t):=(1+t)^a e^{-\frac{Ct}{1+t}} Y_s(t)$, inequality (\ref{eq:Ineg}) gives
	\begin{align*}
		\frac d{dt}& Z \leq Ce^\frac{Ct}{1+t}\frac{ Z^2}{(1+t)^{\Gamma_0+a}}
		+Ce^\frac{(m-1)Ct}{1+t}\frac{ Z^m}{(1+t)^{\Gamma_1+am-1}}\\
		& +Ce^\frac{mCt}{1+t}\frac{ Z^{m+1}}{(1+t)^{\Gamma_2+am-1}}
		+Ce^\frac{(m+1)Ct}{1+t}\frac{ Z^{m+2}}{(1+t)^{\Gamma_4+am-1}},
	\end{align*}
	% $$
	% \frac d{dt}Z
	% \leq
	% Ce^\frac{Ct}{1+t}\frac{ Z^2}{(1+t)^{\Gamma_0+a}}
	% +Ce^\frac{(m-1)Ct}{1+t}\frac{ Z^m}{(1+t)^{\Gamma_1+am-1}}
	% $$
	% $$
	% +Ce^\frac{mCt}{1+t}\frac{ Z^{m+1}}{(1+t)^{\Gamma_2+am-1}}
	% +Ce^\frac{(m+1)Ct}{1+t}\frac{ Z^{m+2}}{(1+t)^{\Gamma_4+am-1}},
	% $$
	where $m=\frac{2}{\gamma-1}-1$.

	%%%%%%%%%%%%%%%%%%%%%%%%%%%%%%%%%%%%%%%%%%%%%%%%%%%%%
	
	In order to prove that $t\to Z(t)$ is bounded for any $t$, we use the same bootstrap argument as in \cite{bddn}.
	
	Suppose indeed that $Z_0:=Z(0)$ and  assume that 
	\begin{equation}\label{eq:Z2}
		Z(t)\leq 2Z_0\quad\hbox{on}\quad [0,T].
	\end{equation}
	Then \eqref{eq:Ineg} implies that
	$$
	\frac d{dt} Z
	\leq
	Ce^{C}\frac{ (2Z_0)^2}{(1+t)^{\Gamma_0+a}}
	+Ce^{(m-1)C}\frac{ (2Z_0)^m}{(1+t)^{\Gamma_1+am-1}}$$
	$$+Ce^{mC}\frac{ (2Z_0)^{m+1}}{(1+t)^{\Gamma_2+am-1}}
	+Ce^{(m+1)C}\frac{ (2Z_0)^{m+2}}{(1+t)^{\Gamma_4+am-1}}.
	$$
	Suppose that
	\begin{equation}
		\label{eq:constraint}
		\Gamma_0+a>1,\ \Gamma_1+am-1>1,\ \Gamma_2+am-1>1\ \mbox{and}\ \Gamma_4+am-1.
	\end{equation}
	Hence, integrating in time, we discover that on $[0,T],$ we have
	$$
	Z(t)\leq 
	Z_0
	+\frac{Ce^C}{\Gamma_0+a-1}(2Z_0)^2\bigl(1-(1+t)^{1-\Gamma_0-a}\bigr)
	+\frac{Ce^{(m-1)C}(2Z_0)^{m+1}}{\Gamma_1+am-2}\bigl(1-(1+t)^{2-\Gamma_1-am}\bigr)
	$$
	$$
	+\frac{Ce^{mC}(2Z_0)^{m}}{\Gamma_1+am-2}\bigl(1-(1+t)^{2-\Gamma_2-am}\bigr)
	+\frac{Ce^{(m+1)C}(2Z_0)^{m+2}}{\Gamma_4+am-2}\bigl(1-(1+t)^{2-\Gamma_2-am}\bigr)
	$$
	Let us discard the obvious case $Z_0=0.$ Then, if  $Z_0$ is so small as to satisfy 
	\begin{equation}
		\label{eq:Bound}
		\frac{4Ce^CZ_0}{\Gamma_0+a-1}
		+\frac{Ce^{(m-1)C}2^{m+1}Z_0^{m}}{\Gamma_1+am-2}
		+\frac{Ce^{mC}2^{m}Z_0^{m-1}}{\Gamma_2+am-2}
		+\frac{Ce^{(m+1)C}2^{m+2}Z_0^{m+1}}{\Gamma_4+am-2}
		\leq 1,
	\end{equation}
	the above inequality ensures that  we actually have $Z(t)<2Z_0$ on $[0,T]$.
	
	Therefore the supremum of $T>0$ satisfying \eqref{eq:Z2} has to be infinite.
	
	%%%%%%%%%%%%%%%%%%%%%%%%%%%%%%%%%%%%%%%%%%%%%%%%%%%%%%%%%%
	Eventually we get, provided $\|(\rho_0,w_0,e_0,b_0)\|_{H^s}$ is small enough:
	\begin{equation}
		\label{eq:poissonHs}
		\sqrt{(1+t)^{2s}\|(\rho,w,e,b)\|_{\dot H^s}^2+\|(\rho,w,e,b)\|_{L^2}^2}\leq 2\frac{e^{\frac{Ct}{1+t}}}{(1+t)^{c_{\gamma}}} \|(\rho_0,w_0,e_0,b_0)\|_{H^s}.
	\end{equation}

	Let us emphasize that in order to derive (\ref{eq:Bound}), we need  to satisfy constraints (\ref{eq:constraint}).
	
	The first condition reduces to $\gamma>1$ and the remaining conditions are equivalent to
	$$
	2+\biggl(\frac2{\gamma-1}-1\biggr)\biggl(a-c_{\gamma}-\frac 32\biggr)<a\biggl(\frac2{\gamma-1}-1\biggr),
	$$
	that is to say
	$$
	\min\biggl(1,\frac32(\gamma-1)\biggr)\biggl(\frac2{\gamma-1}-1\biggr)>2.
	$$
	That latter inequality is equivalent to $\gamma<5/3.$

	Keeping in mind the previous constraint (\ref{eq:Inegs}) one can conclude that \eqref{eq:poissonHs} holds true whenever 
	\begin{equation}
		\label{eq:condP}
		1<\gamma<\frac 53\andf\frac 52<s<\frac12+\frac2{\gamma-1}\cdotp
	\end{equation}
	
	%%%%%%%%%%%%%%%%%%%%%%%%%%%%%%%%%%%%%%
	
	\section{Proving Theorem \ref{isentro}}\label{s:exist}
	
	As pointed out in the introduction, a number of works have been dedicated 
	to the well-posedness issue for the Euler-Maxwell system. 
	However, none of them considered data like ours.
	For the reader's convenience, we here sketch the proof of the global existence
	in the functional setting of Theorem \ref{isentro}, 
	then establish  uniqueness by means of a classical energy method.
	
	\subsection{Existence}

	Here we are given $(\rho_0,u_0,E_0,B_0)$ fulfilling the assumptions of Theorem \ref{isentro}. 
	Our goal is to prove the existence of a global-in-time solution $(\rho,u,E,B)$ for System \eqref{eq:MP} 
	or, equivalently, denoting $w:=u-v,$ of $(\rho,w,e,b)$ for System \eqref{eq:BB}. 
	The idea is to use the cut-off function $\chi$ with range $[0,1]$, support in the ball $B(0,2)$ and value $1$ on $B(0,1)$, and to 
	approximate \eqref{eq:BB} as follows:
	$$\left\{\begin{aligned}
		&(\d_t+u\cdot\nabla)\rho+\frac{d\wt\gamma\rho}{1+t} 
		+\wt\gamma\frac{\rho\,\Tr K_n}{(1+t)^2}
		+\wt\gamma\rho\,\div w=0,\\
		&(\d_t+u\cdot\nabla)w+\frac{w}{1+t} +\frac{w\cdot K_n}{(1+t)^2}
		+\wt\gamma\,\rho\nabla\rho
		%+\eta w
		=-\chi(n^{-1}(\rho E+J\times B)),\\
		&d_t e-\Curl b=-\chi(n^{-1}J),\\
		&d_t b+\Curl e=0,\\
		&(\rho,w,e,b)|_{t=0}=(\rho_0^n,u_0^n,e_0^n,b_0^n),
	\end{aligned}\right.$$
	where $\tilde\gamma=\frac{\gamma-1}{2}$, with $K_n:=\chi(n^{-1}\cdot)K,$  $\rho_0^n:=\chi(n^{-1}\cdot)\rho_0$, $u_0^n:=\chi(n^{-1}\cdot) u_0$,  
	$e_0^n:=\chi(n^{-1}\cdot) e_0$ and $b_0^n:=\chi(n^{-1}\cdot) b_0$.
	\medbreak
	Since the initial data as well as $K_n$ are in the Sobolev space $H^s$ with $s>1+d/2,$ 
	a tiny modification of the standard theory 
	of symmetric hyperbolic systems allows to prove that there 
	exists a unique maximal solution $(\rho^n,w^n,e^n,b^n)$ in 
	$ \cC([0,T^n);H^s)\cap \cC^1([0,T^n);H^{s-1})$ to the above system. 
	
	The computations of the previous step may be repeated on $[0,T_n)$ and one
	gets, with obvious notation, for some absolute constant $C,$
	$$Y_s^n(t) \leq C\frac{e^{\frac{Ct}{1+t}}}{(1+t)^{1+d\wt\gamma}}  Y_s(0)\quad\hbox{for all }\  0\leq t<T_n.$$
	This in particular provides a control on  $\|\nabla \rho^n,\nabla w^n,\nabla e^n,\nabla w^n\|_{L^\infty}$
	so that the classical blow-up criterion for hyperbolic systems allows to conclude that $T_n=+\infty.$
	
	{}From that point, classical functional analysis arguments allow
	to pass to the limit (up to subsequence) and to conclude
	that $(\rho^n,w^n,e^n,b^n)$ converges 
	to some solution $(\rho,w,e,b)$ of \eqref{eq:BB} corresponding to data 
	$(\rho_0,w_0,e_0,b_0).$ Of course, that solution fulfills \eqref{eq:poissonHs}, 
	and looking at the definition of $Y_s$ allows to get the required decay
	estimates.
	
	This completes the proof of the existence part of Theorem \ref{isentro}, and 
	of the decay estimates.

	%%%%%%%%%%%%%%%%%%%%%%%%%%%%%%%%%%%%%%%%
	
	\subsection{Uniqueness}
	
	Consider two solutions $(\rho_1,w_1,e_1,b_1)$ and $(\rho_2,w_2,e_2,b_2)$ of \eqref{eq:BB}
	corresponding to the same data and having the properties of regularity listed in Theorem \ref{isentro}.  Then,  $(\dr,\dw,\delta e,\delta b):=(\rho_2-\rho_1,w_2-w_1,e_2-e_1,b_2-b_1)$ fulfills:
	$$
	\left\{\begin{array}{l} 
		(\d_t+w_2\cdot\nabla)\dr+\wt\gamma\rho_2\,\div\dw+v\cdot\nabla\dr+\wt\gamma\dr\,\div v=
		-\dw\cdot\nabla\rho_1-\wt\gamma\dr\,\div w_1,\\[1.5ex]
		(\d_t+w_2\cdot\nabla)\dw+\wt\gamma\rho_2\,\nabla\dr+v\cdot\nabla\dw
		%+\eta\dw
		\\ \hspace{3cm}=-\delta\left(E+u\times B-v\times \Curl v\right),\\[1.5ex]
		\d_t \delta e-\Curl \delta b+\alpha_1 e=\delta \left(\rho^{\frac{2}{\gamma-1}}u\right),\\[1.5ex]
		\d_t \delta b+\Curl \delta e+\alpha_2 b=0.
	\end{array}
	\right.
	$$
	Hence, differentiating with respect to $x_j$ yields 
	$$\left\{
	\begin{array}{l}
		(\d_t+(v+w_2)\cdot\nabla)\d_j\dr+\wt\gamma\rho_2\,\div\d_j\dw
		=-\d_jw_2\cdot\nabla\dr-\wt\gamma\d_j\rho_2\div\dw-\d_jv\cdot\nabla\dr\\-\wt\gamma\d_j\dr\,\div v-\wt\gamma\dr\,\d_j\div v
		-\d_j\dw\cdot\nabla\rho_1-\dw\cdot\nabla\d_j\rho_1-\wt\gamma\dr\,\div\d_j w_1-\wt\gamma\d_j\dr\,\div w_1,\\[1.5ex]
		(\d_t+(v+w_2)\cdot\nabla)\d_j\dw+\wt\gamma\rho_2\,\nabla\d_j\dr
		%+\eta\d_j\dw
		=-\d_jw_2\cdot\nabla\dw-\wt\gamma\d_j\rho_2\nabla\dr-\d_jv\cdot\nabla\dw\\-\d_j\dw\cdot\nabla v-\dw\cdot\nabla\d_jv
		-\d_j\dw\cdot\nabla w_1-\dw\cdot\nabla\d_jw_1 -\wt\gamma\d_j\dr\nabla\rho_1 -\wt\gamma\dr\nabla\d_j\rho_1\\
		-\d_j\delta e-\d_j\delta (u\times B)+\d_j\delta (v\times \Curl v),\\[1.5ex]
		\d_t \d_j\delta e-\Curl \d_j\delta b+\alpha_1 \d_je=\d_j\delta \left(\rho^{\frac{2}{\gamma-1}}u\right),\\[1.5ex]
		\d_t \d_j\delta b+\Curl \d_j\delta e+\alpha_2\d_j b=0.
	\end{array}\right.
	$$
	Hence, applying an energy method and arguing exactly as for the proof of uniqueness  in the previous section, we get:
	$$
	\frac d{dt}\|(\nabla\dr,\nabla\dw,\nabla\delta e,\nabla\delta e)\|_{L^2}^2
	\lesssim \Bigl(\|(\nabla\rho_1,\nabla\rho_2,\nabla u_1,\nabla u_2,\nabla v,\nabla e_1,\nabla e_2,\nabla b_1,\nabla b_2)\|_{L^\infty}$$
	$$
	+ \|(\nabla^2\rho_1,\nabla^2\rho_2,\nabla^2 u_1,\nabla^2 u_2,\nabla^2v,
	\nabla^2 e_1,\nabla^2 e_2,\nabla^2 b_1,\nabla^2 b_2)\|_{L^d}\Bigr)
	$$
	$$
	\times\|(\nabla\dr,\nabla\dw,\nabla \delta e,\nabla \delta b)\|_{L^2}^2.
	$$
	Recall that $\nabla v$ is bounded and that $\nabla^2v$ is in $H^{s-1}$ with $s>1+d/2,$  and thus in $L^d.$
	
	Of course, as previously $\nabla\rho_i,\nabla w_i,\nabla e_i,\nabla b_i$ are in $L^\infty$
	and $\nabla^2\rho_i,\nabla^2 w_i,\nabla^2 w_i,\nabla^2 e_i,\nabla^2 b_i$ are in $L^d$, for $i=1,2$. 
	Hence Gronwall lemma ensures 
	that $(\nabla\dr,\nabla\dw,\nabla\delta e,\nabla\delta b)\equiv0$ on $[0,T]\times\R^3,$
	which, owing to the fact that $\dr$ is in $L^q$  eventually  implies that $\dr\equiv0.$
	Plugging that information in the equation of  $\dw,$ one can then conclude that $\dw\equiv0.$ 
	Finally one sees in the same stroke that $\delta e=\delta b=0$.
	
	This completes the proof of the theorem. \qed
	
	%%%%%%%%%%%%%%%%%%%%%%%%%%%%%%%%%%%%%%%%%%%%%%%

	\subsection* {Acknowledgments:}
	The first author  warmly thanks D. Serre for a fruitful correspondence on Euler-Maxwell equations.
	\v S.N. and J.S. acknowledge the supports the Praemium Academiae of {\v{S}}. Moreover, \v S.N. acknowledge the supports of  the Czech Science Foundation (GA\v CR) through projects project 22-01591S. The Institute of Mathematics, CAS is supported by RVO:67985840. 
	%His research is partially supported by the ANR project INFAMIE (ANR-15-%CE40-0011).

	\vskip.25cm
	
	\centerline{Bernard Ducomet}
	\centerline{Universit\'e Paris-Est}
	\centerline{LAMA (UMR 8050), UPEMLV, UPEC, CNRS}
	\centerline{ 61 Avenue du G\'en\'eral de Gaulle, F-94010 Cr\'eteil, France}
	\centerline{E-mail: bernard.ducomet@u-pec.fr}
	\vskip0.5cm
	\centerline{\v S\' arka Ne\v casov\'a}
	\centerline{Institute of Mathematics of the Academy of Sciences of the Czech Republic}
	\centerline{\v Zitna 25, 115 67 Praha 1, Czech Republic}
	\centerline{E-mail: matus@math.cas.cz}
	\vskip0.5cm
	\centerline{John Sebastian H. Simon}
	\centerline{Johann Radon Institute for Computation and Applied Mathematics (RICAM)}
	\centerline{Austrian Academy of Sciences}
	\centerline{Altenberger Strasse 69, 4040 Linz, Austria}
	\centerline{E-mail: john.simon@ricam.oeaw.ac.at, jhsimon1729@gmail.com}

\end{document}